\documentclass[11pt]{article}
\usepackage{graphicx}
\usepackage{subfig}
\usepackage{tabularx}

\usepackage{mathrsfs}     % selects Times Roman as basic font
\usepackage{helvet}         % selects Helvetica as sans-serif font
\usepackage{courier}        % selects Courier as typewriter font
\usepackage{type1cm}      % activate if the above 3 fonts are
\usepackage{color}
\usepackage{todonotes}
\usepackage{mathbbol}

\usepackage{geometry,calc,color}
\usepackage{mathtools,amssymb}

\usepackage{enumerate}

\let\oldlangle\langle
\let\oldrangle\rangle
\usepackage{mathabx,stmaryrd}

\usepackage[b]{esvect}
\usepackage[utf8]{inputenc}
\usepackage{amsthm}

\newtheorem{theorem}{Theorem}[section]
\newtheorem{corollary}[theorem]{Corollary}
\newtheorem{lemma}[theorem]{Lemma}
\newtheorem{proposition}[theorem]{Proposition}

\newtheorem{comment}[theorem]{Comment}

\theoremstyle{remark}
\newtheorem{remark}[theorem]{Remark}

\numberwithin{equation}{section}
\date{September 4, 2021}

\newcommand{\uno}{}
\newcommand{\menouno}{_*}

\newcommand{\basis}{\mathfrak{B}}
\newcommand{\BaseQ}{\mathfrak{B}_Q}
\newcommand{\basf}{\widecheck e}
\newcommand{\BaseU}{\mathfrak{B}_U}
\newcommand{\basv}{\widehat e}

\renewcommand{\u}{{\mathbf u}}
\renewcommand{\v}{{\mathbf v}}
\newcommand{\w}{{\mathbf w}}

\newcommand{\te}{\eta}
\newcommand{\e}{e}

\newcommand{\V}{\mathcal V}

\newcommand{\bil}{\frak{a}}
\newcommand{\A}{\mathfrak{A}}

\newcommand{\tW}{W^*}

\newcommand{\normP}{\| P \|}

\newcommand{\dualscal}[2]{(#1,#2)\menouno}
\newcommand{\matr}[1]{\mathbf{#1}}

\newcommand{\astab}{\bil_{\gamma}}
\newcommand{\gstab}{G_{\gamma}}

\newcommand{\hag}{\widehat \bil_\gamma}
\newcommand{\hbg}{\widehat \beta_\gamma}

\renewcommand{\langle}{\oldlangle}
\renewcommand{\rangle}{\oldrangle}
\renewcommand{\vec}[1]{\vv{#1}}

\newcommand{\ucont}{{u^\flat}}
\newcommand{\pcont}{{p^ \flat}}
\newcommand{\wcont}{{w^\flat}}
\newcommand{\uh}{{u_h^\flat}}
\newcommand{\ph}{{p_h^\flat}}

\renewcommand{\H}{\mathcal{H}}
\newcommand{\U}{U}
\newcommand{\Q}{Q}

\newcommand{\infsupconst}{\beta}

\newcommand{\cinfsup}{\infsupconst}

\newcommand{\VV}{\mathscr{V}}

\newcommand{\cstar}{c_\star}
\newcommand{\Cstar}{C_\star}
\newcommand{\comput}{{\mathbf c}}

\newcommand{\ccoerc}{\alpha}

\newcommand{\cuno}{\alpha} % coercivity constant for a
\newcommand{\kappastar}{\widehat c}
\newcommand{\Kstar}{\widehat C}

\newcommand{\contS}{K_\star}
\newcommand{\coercS}{\kappa_\star}

\begin{document}

\title{Algebraic representation of dual scalar products and stabilization of saddle point problems}%
\author{Silvia Bertoluzza
\thanks{{This paper has been partially supported by the ERC Project CHANGE, which has received funding from the European Research Council (ERC) under the European Union’s Horizon 2020 research and innovation programme (grant agreement No 694515), and was co-funded by the MIUR Progetti di Ricerca di Rilevante Interesse Nazionale (PRIN) Bando 2017 (grant 201744KLJL).}}
}

%% Author for style amsart
%\author{Silvia Bertoluzza
%}	\thanks{{This paper has been partially supported by the ERC Project CHANGE, which has received funding from the European Research Council (ERC) under the European Union’s Horizon 2020 research and innovation programme (grant agreement No 694515)}}

%\address{Istituto di Matematica Applicata e Tecnologie Informatiche del CNR }%
%\email{silvia.bertoluzza@imati.cnr.it}%
%
%%\thanks{}%
%\subjclass{}%
%\keywords{Residual based stabilization, non coercive problems, saddle point problem, dual scalar products}%

\date{\today}
%
%\dedicatory{}%
%\commby{}%
% ----------------------------------------------------------------

\maketitle

\begin{abstract}
	We provide a systematic way to design computable bilinear forms which, on the class of subspaces $W^* \subseteq \V'$ that can be obtained by duality from a given finite dimensional subspace $W$ of an Hilbert space $\V$, are spectrally equivalent to the scalar product of $\V'$.
	In the spirit of \cite{BBstab,Babs2}, such bilinear forms can be used to build a stabilized discretization algorithm for the solution of an abstract saddle point problem allowing to decouple, in the choice of the discretization spaces, the requirements related to the approximation from the ones related to the inf-sup compatibility condition, which, however, can not be completely avoided.
\end{abstract}

\section{Introduction}\label{sec:intro}
	This work is dedicated to the memory of Claudio Baiocchi, whom I was lucky enough to have as supervisor at the beginning of my journey as a researcher, and to whom I am thankful, for the profound influence he had on my way of looking at mathematical problems. 
	
	\

We are interested in problems of the form
\begin{multline}\label{contsaddle}
\text{find $(\ucont,\pcont) \in \V \times \H$ such that $\forall (v,q) \in \V \times \H$}\\ a(\ucont,v) - b(\pcont,v) + b(q,\ucont) = \langle F , v \rangle + \langle G,q \rangle,
\end{multline} 
where $\V$ and $\H$ are given Hilbert spaces, $a : \V \times \V \to \mathbb{R}$, $b : \H \times \V \to \mathbb{R}$ are two bounded bilinear forms, and $F\in \V'$, $G \in \H'$ the given data. 
Equations in this form arise when using a Lagrange multiplier approach for the solution of constrained optimization problems and can be encountered in many application areas, in fields such as engineering, mathematical physics, numerical analysis, just to name a few. 
It is well known  (see \cite{BoffiBrezziFortin}) that the discretization of such problems requires either a careful choice of the approximation spaces, that need to satisfy a sometimes restrictive {\em inf-sup} condition, or the introduction of some form of regularization. 
More precisely, assuming, for the sake of simplicity, that the bilinear form $a$ is coercive on $\V$, the well posedness of this class of problems relies on the validity of the inf-sup condition
\begin{equation}\label{firstinfsup}
\inf_{p\in \H} \sup_{v\in \V} \frac{b(p,v)}{\vvvert p \vvvert\, \| v \|} \geq \infsupconst > 0,
\end{equation}
where $\vvvert\cdot\vvvert$ and $\| \cdot \|$ denote, respectively, the norms of $\H$ and $\V$. Unfortunately, the validity of condition \eqref{firstinfsup} is not automatically inherited by generic finite dimensional subspaces of $\V$ and $\H$ and, if these are not properly chosen, it is well known that  stability issues may arise when considering the Galerkin discretization of a problem of the type considered. Ensuring the discrete equivalent of \eqref{firstinfsup} is not always easy, and it generally requires the use of unequal order polynomial approximations (see e.g. \cite{BolandNicolaides1985}) and/or different grids for the two unknowns (see e.g. \cite{Babuska}), sometimes also involving non trivial constructions for the discretization spaces. Alternatively, for different class of problems, several recipes have been proposed in the literature for circumventing the inf-sup condition, by either enriching the space $\V$ (see, e.g., \cite{MINI,BFMR}), or by suitably modifying the discrete problem, as in \cite{SUPG} (actually, the two approaches turn out to be strictly related to one another, see e.g. \cite{VBG}).

More in general, a similar stability issue occurs for a wider class of equations,  falling in the abstract unifying framework considered in \cite{BBstab} (see also \cite{BoffiBrezziFortin,Bochevetal2006}), where,  for positive semidefinite problems on an abstract Hilbert space $\VV$, a stabilization approach is proposed, based on adding a consistent residual term, measured in the norm of $\VV'$. Unfortunately, handling in practice such a residual term in the implementation would require numerically evaluating the scalar product for $\VV'$, which might be difficult and/or expensive, and, in \cite{BBstab}, the authors introduce the abstract stabilized  problem (see \eqref{contstab} in the next section) mainly  with the aim of giving a unified interpretation to a class of well established stabilized methods, including {\em Streamline Upwind Petrov Galerkin} (SUPG, \cite{SUPG}) and {\em Galerkin Least Squares} (GLS, \cite{GLS}). Indeed, one can interpret the suitably scaled mesh dependent scalar products appearing in such stabilized methods as a way of mimicking the scalar products of the dual spaces where the concerned residuals naturally ``live''.
 However,  actually using the scalar product for $\VV'$ (or a spectrally equivalent bilinear form) in designing  stabilized methods 
might prove useful in those cases where the standard mesh dependent stabilizing bilinear form  yields suboptimal results, as it may happen, for instance, when dealing with meshes with degrading shape regularity, or when aiming at robust $hp$ type estimates. The use of a dual space scalar product has already been  explored in different contexts,  numerically realizing it by resorting to wavelets \cite{BCT,BK,B:Lag}, to preconditioners for the stiffness matrix relative to a coercive bilinear form on $\V$ (usually the Sobolev space $H^1$ \cite{Bramble98,Bramble01,Arioli}), or by locally approximating the inversion of the Riesz operator in a suitably enriched space (as it happens, for instance, in the Discontinuous Petrov Galerkin mehod \cite{DPG2,DPG} or when stabilizing the Stokes discretization in the reduced basis method \cite{RozzaSupremizer}). 

The aim of the present paper is to provide a systematic way of realizing, in an abstract framework  that can be easily fitted to concrete situations,  computable bilinear forms which are, on given finite dimensional spaces, spectrally equivalent to the scalar product in ${\VV'}$, thus providing a flexible tool for designing robust stabilized methods for saddle point problems. This tool has already been exploited in \cite{DGPoly1,Bertoluzza2020probust} and \cite{bertoluzza2021stabilization}, where the abstract approach here proposed has been successfully applied, respectively, in the context of the polygonal discontinuous Galerkin method, and of the non conforming virtual element method.

\

The paper is organized as follows. In Section \ref{sec:1}, for the sake of completeness, we recall the abstract result of \cite{BBstab}. In Section \ref{sec:2}, we apply to the stabilization of saddle point problems the idea presented in Section \ref{sec:1}, and leverage the coercivity of the bilinear form $a$ to derive a simplified version of the stabilization term. In Section \ref{dualscalprod} we address the problem of  designing (equivalent) scalar products for suitable finite dimensional subspaces of the dual space involved in the definition of the stabilization term; in Section \ref{sec:5} we combine the results of Sections \ref{sec:2} and \ref{dualscalprod} to obtain a computable stabilized discretization of the saddle point problem \eqref{contsaddle}. Moreover we give a criterion for the choice of the auxiliary space $W$ which underlies the definition of the computable stabilization term. Finally, in Section \ref{sec:decoupling} we give a different interpretation of the role of the auxiliary space $W$. Section \ref{sec:conclusions} presents some final comments.

\newcommand{\Cbil}{C_{\bil}}

\section{Stabilization in an abstract framework}\label{sec:1}
For the sake of completeness we start by recalling the abstract result presented in \cite{BBstab} (see also \cite{BoffiBrezziFortin} for additional details). Let $\VV$ be a given Hilbert space endowed with the norm  $\| \cdot \|_\VV$, and let $\bil : \VV \times \VV \to \mathbb{R}$ be a bilinear form satisfying the continuity bound
\begin{equation}\label{continuity}
\bil(\v,\w) \leq \Cbil \| \v \|_{\VV} \| \w \|_{\VV}, \quad \forall \v,\w \in \VV.
\end{equation}
We assume $\bil$ to be positive semi-definite, that is
\begin{equation}\label{3.1}
\bil(\v,\v) \geq 0, \quad \forall \v \in \VV.
\end{equation}

\newcommand{\uc}{\u^\flat}

We consider a problem of the form
\begin{equation}\label{abstractpb}
\text{find $\uc \in \VV$ such that $ \forall \v \in \VV$}\qquad \bil(\uc,\v) = \langle \mathscr{F},\v \rangle,
\end{equation}
and its Galerkin discretization
\begin{equation}\label{galerkin}
\text{find $\uc_h \in \VV_h$ such that $ \forall \v \in \VV_h$}\qquad \bil(\uc_h,\v) = \langle \mathscr{F},\v \rangle,
\end{equation}
where $ \mathscr{F}$ is a given functional in $\VV'$, and where $\VV_h \subset \VV$ is a finite dimensional subspace of $\VV$.
Assuming that Problem \eqref{abstractpb} is well posed, that is, that the linear functional 
 $\A: \VV \to \VV'$, induced by the bilinear form $\bil$ and defined as
 \begin{gather*}
 \langle \A \v,\w \rangle = \bil(\v,\w), \quad\forall \w \in \VV,
 \end{gather*}
  is an isomorphism of $\VV$ onto $\VV'$, we are interested in the well posedness of the discrete problem \eqref{galerkin}.
  
  \
  
   A standard sufficient condition for $\A$ to be an isomorphism of $\VV$ onto $\VV'$ is that the bilinear form $\bil$ is coercive, that is, that it satisfies, for some constant $\alpha_0 > 0$,
 \begin{equation}\label{coercivity}
 \bil(\v,\v) \geq \alpha_0 \| \v \|_{\VV}^2, \qquad \forall \v \in \VV.
 \end{equation}
Of course, condition \eqref{coercivity} is inherited by $\VV_h$. Consequently,  if \eqref{coercivity} holds, Problem \eqref{galerkin} is itself well posed, and, by C\'ea's Lemma \cite{Cea}, its solution $\uc_h$ is a quasi-best approximation to $\uc$:
\begin{equation}\label{errorbound}
 \| \uc - \uc_h \|_{\VV} \leq \frac {C_\frak{a}} {\alpha_0} \inf_{\v \in \VV_h} \| \uc - \v \|_{\VV}.
\end{equation}
 
On the other hand, for many well posed problems of the form \eqref{abstractpb}, the bilinear form $\bil$ does not satisfy \eqref{coercivity}, but only  the following weaker condition (see \cite{BabuskaAziz}): \begin{equation}\label{doubleinfsup}
 \inf_{\v \in \VV} \sup_{\w \in \VV} \frac{\bil(\v,\w)}{\| \v \|_\VV \| \w \|_\VV } \geq \alpha_0 >0 \quad \text{and}\quad \sup_{\v \in \VV} {\bil(\v,\w)} > 0, \quad \forall \w \in \VV.
 \end{equation}
 
 In such a case, depending on the choice of the discrete subspace $\VV_h$, problem \eqref{galerkin} might be numerically  unstable (that is, it might satisfy a stability bound with a very large constant), or even not admit a unique solution at all. To remedy this, \cite{BBstab} presents a general strategy for defining a new equivalent problem, also in the form \eqref{abstractpb}, for which the bilinear form is coercive.
 
 In order to do so, let
 $\A_s: \VV \to \VV'$ and $\A_a:\VV \to \VV'$ respectively denote the linear operator induced by the  symmetric and antisymmetric part of the bilinear form $\bil$:
\begin{gather*}
\langle \A_s \v , \w \rangle = \frac{\bil(\v,\w) + \bil(\w,\v)}{2},\quad \forall \w \in \VV,\\[4mm] \langle \A_a \v , \w \rangle = \frac{\bil(\v,\w) - \bil(\w,\v)}{2},\quad \forall \w \in \VV.
\end{gather*}
We can then consider the following problem:
\begin{multline}
\text{find $\uc \in \VV$ such that  $\forall \v \in \VV$
	}\\ \bil(\uc,\v) + \gamma( \A \uc, \A_t \v )_{\VV'}= \langle  \mathscr{F},\v \rangle+ \gamma( \mathscr{F},\A_t \v)_{\VV'}, \label{contstab}
\end{multline}
where $t \in \mathbb{R}$ is a scalar parameter, arbitrary but fixed (usually $t\in \{-1,0,1\}$), $\A_t$ is defined as 
\begin{equation}\label{defAt}
\A_t = \A_a + t \A_s,
\end{equation}
$\gamma$ is a positive constant, and where $(\cdot,\cdot)_{\VV'}$ denotes the scalar product in $\VV'$. The following theorem, of which, for completeness, we report the proof (see \cite{BBstab,BoffiBrezziFortin}), holds.
\begin{theorem}\label{teoremabb} Assume that \eqref{continuity}, \eqref{3.1} and \eqref{doubleinfsup} hold. Then for any $t \in \mathbb{R}$, $\gamma > 0$ with $\gamma(t - 1)^2 < 4M^{-1}$, it holds that
	\begin{equation}\label{coercTh1.1}
	\bil(\v,\v) + \gamma(\A \v,\A_t \v)_{\VV'} \geq \beta_\gamma \| \v \|^2_\VV,
	\end{equation}
	with $\beta_\gamma > 0$ depending on $t$ and $\gamma$, as well as on $\alpha_0$ and $\Cbil$.
\end{theorem}	

\begin{proof} The continuity \eqref{continuity} and non negativity \eqref{3.1} of the bilinear form $\bil$ imply that, for all $\u \in \VV$,
	\begin{equation}\label{contsimm}
	\| \A_s \u \|_{\VV'} \leq \sqrt{\Cbil \bil(\u,\u)}.
	\end{equation}
Indeed, \eqref{3.1} implies that the symmetric bilinear form $\bil_s: \VV \times \VV \to \mathbb{R}$ defined as
\[
\bil_s(\v,\w)  = \langle \A_s \v, \w \rangle,
\] 
satisfies a Schwarz inequality. Using \eqref{continuity}, we can then write
\begin{multline*}
| \langle \A_s \v, \w \rangle | = | \bil_s(\v,\w) | \leq \sqrt{\bil_s(\v,\v)} \sqrt{\bil_s(\w,\w)} \\ =  \sqrt{\bil(\v,\v)} \sqrt{\bil(\w,\w)} \leq  \sqrt{\bil(\v,\v)} \sqrt{C_\frak{a}}  \| \w \|_{\VV}.
\end{multline*}
Taking the supremum over $\w \in \VV$ we get  \eqref{contsimm}. 
Then, for $\varepsilon > 0$ arbitrary, we have
\begin{multline}
(\A \v,\A_t \v)_{\VV'} = \| \A_a \v \|_{\VV'}^2 + (1+t)(\A_a \v,\A_s \v )_{\VV'} +t \| \A_s \v \|^2_{\VV'} \\[1mm]
\geq \| \A_a \v \|_{\VV'}^2 - | 1 + t | \| \A_a \v \|_{\VV'}  \| \A_s \v \|_{\VV'} +  t  \| \A_s \v  \|_{\VV'}^2 \\[1mm]
\geq \Big(1 - |1+t| \frac{\varepsilon}{2}\Big)  \| \A_a \v \|_{\VV'}^2  + \Big(t - \frac{| 1 + t|}{2\varepsilon}\Big) \| \A_s \v  \|_{\VV'}^2.
\end{multline}
Combined with \eqref{contsimm}, this implies that
\begin{multline}\label{combined}
	\bil(\v,\v) + \gamma(\A \v,\A_t \v)_{\VV'} \geq \Big(
\frac 1 {\Cbil} + \gamma t - \gamma \frac{| 1 + t |}{2\varepsilon}
	\Big)  \| \A_s \v \|^2_{\VV'} \\+ \gamma \Big( 1 - |1 + t | \frac{\varepsilon}{2} \Big)  \| \A_a \v \|_{\VV'}^2.
\end{multline}
A simple calculation allows to check that the coefficients of $\| \A_s \v \|^2_{\VV'}$ and $\| \A_a \v \|^2_{\VV'}$ at the right hand side of \eqref{combined} are both strictly positive if and only if
\[
\frac {\gamma | 1+t |}{2(\Cbil^{-1} + \gamma t)} < \varepsilon < \frac 2 {| 1 + t |}.
\] 
The condition $\gamma(t - 1)^2 < 4\Cbil^{-1}$ in the statement of the theorem guarantees that the lower bound in the above condition is strictly lower than the upper bound, and hence, a strictly positive $\varepsilon$ satisfying such a condition exists, by choosing which in bound \eqref{combined}, we obtain that, for some positive constants $C_1$ and $C_2$ 
\begin{equation}\label{coercstep}
	\bil(\v,\v) + \gamma(\A \v,\A_t \v)_{\VV'} \geq C_1  \| \A_s \v \|^2_{\VV'} + C_2 \| \A_a \v \|_{\VV'}^2.
\end{equation}
We conclude by recalling that condition \eqref{doubleinfsup} implies
\[
\alpha_0  \| \v \|_{\VV} \leq \sup_{\w \in \VV} \frac{\bil(\v,\w)}{\| \w \|_{\VV}} = 
\sup_{\w \in \VV} \frac{\langle \A \v,\w \rangle }{\| \w \|_{\VV}}  \leq
 \| \A \v \|_{\VV'} \leq  
 \| \A_s \v \|_{\VV'} + \| \A_a \v \|_{\VV'},
\]
which, combined with \eqref{coercstep} gives us \eqref{coercTh1.1} for a constant $\beta_\gamma$ depending on $\alpha_0$ as well as on $C_1$, $C_2$, which in turn depend on $\Cbil$, $t$ and $\gamma$. 
	\end{proof}

Theorem \ref{teoremabb} implies that, given $t$, there exists a $\gamma_0$ such that for all $\gamma < \gamma_0$, Problem \eqref{contstab} admits a unique solution. As, clearly, the solution to Problem \eqref{abstractpb} is also a solution to Problem \eqref{contstab}, the two formulations are equivalent. For the discretization of \eqref{abstractpb} we can then apply the Galerkin method to the formulation \eqref{contstab} and consider the well posed discrete problem 
\begin{multline}
\text{find $\uc_h \in \VV_h$ such that  $\forall \v \in \VV_h$
}\\[1mm] \bil(\uc_h,\v) + \gamma( \A \uc_h, \A_t \v )_{\VV'}= \langle  \mathscr{F},\v \rangle+ \gamma( \mathscr{F},\A_t \v)_{\VV'}. \label{discstab}
\end{multline}
Of course, also for the solution of Problem \eqref{discstab} we can apply C\'ea's lemma and we obtained an error bound similar to \eqref{errorbound}, the constants in the bound depending, this time, also on $\gamma$ and $t$.

%	
%	\section{Realizing dual semi inner products}

\newcommand{\CA}{C_a }
\newcommand{\CB}{C_b }
\newcommand{\lscalH}{\llbracket}
\newcommand{\rscalH}{\rrbracket}

\section{Stabilization of saddle point problems}\label{sec:2}

Let us now focus on the case of saddle point problems of the form \eqref{contsaddle}. Using the notation already partially introduced in Section \ref{sec:intro},  $\V$ and $\H$ are Hilbert spaces, $\V'$ and $\H'$ their duals, with
$(\cdot,\cdot)\uno$  and $( \cdot ,\cdot )\menouno$ respectively denoting the scalar products for $\V$ and for $\V'$, $\lscalH \cdot,\cdot \rscalH\uno$  and $\lscalH \cdot ,\cdot \rscalH\menouno$  the scalar products for, respectively, $\H$ and  $\H'$, and with $\langle\cdot,\cdot \rangle$  denoting, by abuse of notation, both the duality between $\V'$ and $\V$ and the duality between $\H'$ and $\H$ (the correct instance will be clear from the context). 
Let  $\| \cdot \| = (\cdot,\cdot)^{1/2}$ (resp. $\vvvert \cdot \vvvert\uno = \lscalH \cdot,\cdot \rscalH\uno^{1/2}$) denote the norm for $\V$ (resp. $\H$) and $\| \cdot \|\menouno = (\cdot,\cdot)\menouno^{1/2}$ (resp. $\vvvert \cdot \vvvert\menouno = \lscalH \cdot,\cdot \rscalH\menouno^{1/2}$)  the norm for $\V'$ (resp. $\H'$).  Finally    $a : \V \times \V \to \mathbb{R}$ and $b : \H \times \V \to \mathbb{R}$ are two bounded bilinear forms, and $F\in \V'$, $G \in \H'$ the given data. 

\

\newcommand{\stab}{\mathfrak{s}}
\newcommand{\Fstab}{\mathscr{G}}

We make some 
 standard assumptions on the two bilinear forms 
$a$ and $b$. We assume that they are continuous,
\begin{gather}
a(v,w) \leq \CA  \| v \|\uno \| w \|\uno, \qquad \forall v,w \in \V, \label{conta}\\[2mm]
b(p,w) \leq \CB \vvvert p \vvvert\uno \| w \|\uno, \qquad \forall p\in \H,~\forall w \in \V,\label{contb} 
\end{gather}
and that they 
respectively satisfy, for $\ccoerc > 0$ and $\cinfsup > 0$ positive constants, a coercivity condition for $a$ on  $\ker b = \{w \in \V: b(q,w) = 0\  \forall q \in \H\} \subset \V$,
\begin{equation}\label{coercker}
a(u,u) \geq \ccoerc \| u \|^2, \qquad \text{for all }u \in \ker b, 
\end{equation}
and a compatibility condition for $b$ on $\H \times \V$, in the form 
\begin{equation}\label{infsupbb}
\inf_{p \in \H} \sup_{v \in \V} \frac{b(p,v)}{\vvvert p \vvvert\, \| v \|} \geq \cinfsup.  \end{equation}
As already observed, Problem \eqref{contsaddle} falls in the framework 
addressed by Theorem \ref{teoremabb}, with $\VV = \V \times \H$ and with, for $v,w \in \V$ and $p,q \in \H$ 
\[ 
\bil(u,p;v,q) = a(u,v) - b(p,v) + b(q,u).
\]
Introducing the operators $A: \V \to \V'$, $B : \H \to \V'$ and $B^T : \V \to \H'$ defined by
\begin{gather*}
\langle Au,v \rangle = a(u,v), \quad \langle Bp, v \rangle = b(p,v), \qquad \forall v\in \V, \\
 \langle B^T u , q \rangle = b(q,u)\qquad \forall  q \in \H, 
\end{gather*}
a plain application of the results of the previous section yields the following equivalent formulation:
\begin{multline}\label{bbsaddle}
\text{find $(\ucont,\pcont) \in \V \times \H$ such that $\forall (v,q) \in \V \times \H$}\\[1mm]
\bil(\ucont,\pcont;v,q) + \gamma \stab(\ucont,\pcont;v,q) = \langle 
F,v \rangle + \langle G, q \rangle + \gamma \Fstab(v,q),
\end{multline}
with 
\begin{gather}\label{deffullstab}
\stab(u,p;v,q) = (A u- Bp, (A_a + t A_s)v - B q )\menouno +  \lscalH B^T u , B^T v \rscalH\menouno\end{gather}
and 
\begin{gather*}\Fstab(v,q) =  (F, (A_a + t A_s)v - B q )\menouno+ \lscalH G , B^T v \rscalH\menouno,
\end{gather*}
where $A_s$ and $A_a$ denote, respectively, the symmetric and antisymmetric part of the operator $A$.
Theorem \ref{teoremabb} states that, for $\gamma$ sufficiently small, the stabilized bilinear form $\bil + \gamma \stab$ is coercive. 
If we then let  $\U \subset \V$ and $\Q \subset \H$ denote finite dimensional subspaces,  in order to obtain a stable discretization of problem \eqref{contsaddle}, independently on whether the spaces $\U$ and $\Q$ satisfy a compatibility condition (the discrete analogous of \eqref{coercker} and  \eqref{infsupbb}), we can apply the Galerkin method to Problem \eqref{bbsaddle}, and this will result in a quasi optimal approximation $(\uh,\ph) \in \U \times \Q$, satisfying the error bound
\[
\| \ucont - \uh \|\uno + \vvvert \pcont - \ph \vvvert\uno \leq C \Big(\inf_{v\in \U} \| \ucont - v \|\uno 
+ \inf_{q\in \Q}  \vvvert \pcont - q \vvvert\uno\Big),
\] 
where the constant $C$ depends on $t$ and $\gamma$.

\

As we will see shortly, it is, however, not difficult to realize that, if we assume $a$ to be coercive on the whole of $\V$ rather than only on $\ker b$, that is, if we assume that
\begin{equation}\label{coerca}
a(u,u) \geq \ccoerc \| u \|^2, \qquad \text{for all }u \in \V, 
\end{equation}
 a stable equivalent formulation of Problem \eqref{contsaddle} is obtained also when replacing the full residual term 
 $\stab(u,p;v,q) - \Fstab(v,q)$
 with a partial residual, namely $\dualscal{A u - Bp - F}{-Bq}$. The discretization of \eqref{contsaddle} would then read
\begin{multline}\label{discretesaddlecont}
\text{find $(\uh,\ph) \in U \times Q$ such that }\\
\astab(\uh,\ph;v,q) = \langle F , v \rangle + \langle \gstab, q \rangle \quad \forall v,q \in \U\times \Q
\end{multline}
with
\begin{gather*}
\astab(u,p;v,q) = \bil(u,p;v,q) - \gamma\dualscal{Au - Bp}{Bq},\\[2mm] \langle \gstab , q \rangle = \langle G ,q\rangle - \gamma\dualscal{F}{Bq} .
\end{gather*}

\

Of course, in general, the scalar products $\dualscal\cdot\cdot$ and $\lscalH \cdot,\cdot \rscalH\menouno$ will not be practically computable.
Focusing, for the sake of simplicity,  on problems for which \eqref{coerca} holds, we replace  the scalar product $(\cdot,\cdot)\menouno$  in the
 discrete Problem \eqref{discretesaddlecont}, with a computable bilinear form  $\comput: \V' \times \V' \to \mathbb{R}$, for which we will assume continuity,
\begin{equation}\label{contcomput}
\comput(\eta , \zeta) \leq \Cstar \| \eta \|\menouno \| \zeta \|\menouno.
\end{equation} 
We then consider the following formulation:
\begin{multline}\label{discretesaddle}
\text{find $(\uh,\ph) \in U \times Q$ such that }\\
\hag(\uh,\ph;v,q) = \langle F , v \rangle + \langle \widehat G_\gamma, q \rangle \quad \forall v,q \in \U\times \Q,
\end{multline}
with 	$\hag: \V \times \V \to \mathbb{R}$ and $\widehat G_\gamma \in \H'$ defined by
\begin{gather*}
\hag(u,p;v,q) = \bil(u,p;v,q) - \gamma \comput(Au - Bp,Bq),\\[2mm]
 \langle \widehat G_\gamma , q \rangle = \langle G, q \rangle  - \gamma \comput(F,B q).
\end{gather*}
To get control on $p$, we  assume the coercivity of $\comput$ in $B(Q) = \{ \eta \in \V': \exists q \in Q \text{ s.t. } \eta = Bq \}$, that is
\begin{equation}\label{star}
\comput(Bq,Bq) \geq \cstar \| B q\|\menouno^2, \qquad \text{for all } q \in \Q,
\end{equation}
for a positive constant $\cstar$.

\

%   and we let $W_h$, $\tW _h$ be the spaces introduced in the previous section. We consider the following stabilized discrete formulation of the problem considered:
%find $u_h \in V_h$, $p_h \in Q_h$ such that for all $v_h \in V_h$, $q_h \in Q_h$ it holds that
%\begin{gather}
%a(u_h,v_h) + b(v_h,p_h) = \langle f , v_h \rangle \label{stab:1}\\[2mm]
%- b(u_h,q_h) + \gamma [\tilde P_h(B^T p_h + A u_h - f),\tilde P_h(B^T q_h)]\menouno = \langle g , q_h \rangle.
%\label{stab:2} \end{gather}
%with
%\[
%[\eta_h,\zeta_h]\menouno = \vec \eta^T S^{-1} \vec \zeta.
%\]
%
%\
%
%In order for Problem (\ref{stab:1}-\ref{stab:2}), $\tW _h$ must be big enough to capture a portion of $(B^T p_h$). This is the meaning of the following assumption:
%\begin{assumption}\label{star}
%	$\tW _h$ is chosen in such a way that for all $q_h \in Q_h$ we have
%	\[
%	[\tilde P_h (B^T q_h),\tilde P_h (B^T q_h)]\menouno \simeq \| \tilde P_h (B^T q_h) \|\menouno^2 \gtrsim \| B^T q_h \|\menouno^2
%	\]
%\end{assumption}

We have the following Theorem
\begin{theorem}\label{teorema} Assume that \eqref{conta}, \eqref{contb}, \eqref{infsupbb}, \eqref{coerca}, as well as \eqref{contcomput} and \eqref{star} hold.
	Then, there exists $\gamma_0 >0$ such that, for all $\gamma < \gamma_0$, the bilinear form $\hag$
	is coercive on $\U \times \Q$, that is there exists $\hbg >0$ such that:
	\begin{equation}\label{coerchag}
	\hag(v,q;v,q) \geq \hbg (\| v \|^2\uno + \vvvert q \vvvert\uno^2) \quad \forall v \in \U, q \in \Q.
	\end{equation}
\end{theorem}
\begin{proof} We start by remarking that, as we assumed that \eqref{conta}--\eqref{infsupbb} hold,
		Problem \eqref{contsaddle} is well posed (see \cite[Theorem 4.2.3]{BoffiBrezziFortin}). Then,
		it is not difficult to prove that there exist two constants $\Kstar \geq \kappastar > 0$ such that for all $v \in \V$ and $q \in \H$ we have
	\begin{equation}\label{a}
	\kappastar(\| v \|^2\uno + \vvvert q
	\vvvert^2\uno) \leq \| v \|\uno^2+ \| B q \|^2\menouno	\leq \Kstar (\| v \|^2\uno + \vvvert q \vvvert^2\uno) .
	\end{equation}
	The constant $\kappastar$ depends only on the continuity constant $\CA$ in \eqref{conta}, the coercivity constant $\alpha$ in \eqref{coerca} and the inf-sup constant $\beta$ in \eqref{infsupbb}, while the constant $\Kstar$ only depends on the continuity constant $\CB$ in \eqref{contb}. 	Let us then show that $\hag(v,q;v,q)$ controls from above the quantity on the right hand side of \eqref{a}. Using \eqref{coerca}, \eqref{contcomput} and \eqref{star}, we can write, for $\varepsilon > 0$ arbitrary,
	\begin{multline*}
	\hag(v,q;v,q) = a(v,v) + \gamma \dualscal{Bq}{Bq} - \gamma \dualscal{Au}{Bq}\\[2mm]
	\geq 
	\cuno \| v \|\uno^2 + \gamma c_\star \| B q \|\menouno^2  - \gamma  C_\star  \| B q \|\menouno \| A v \|\menouno \\[2mm]\geq
	\cuno \| v \|\uno^2 + \gamma (c_\star - C_\star \varepsilon) \| B q \|\menouno^2  - \gamma C_\star \| A v \|^2\menouno /(4\varepsilon)\\[2mm]\geq
	(\cuno -  \gamma  C_\star \CA^2 /(4\varepsilon))\| v \|\uno^2 +  \gamma (c_\star- C_\star \varepsilon) \| B q \|\menouno^2.
	\end{multline*}
	Choosing $\varepsilon = c_\star / (2C_\star)$ and, subsequently, $\gamma_0  = 2 \cuno c_\star / (\CA^2 C^2_\star)$, using \eqref{a}
	we obtain that \eqref{coerchag} holds for 
	\[
	\beta_\gamma = \min \Big\{\gamma \frac {c_\star}2, \cuno - \gamma \frac{C^2_\star \CA^2}{2 c_\star} \Big\} \kappastar > 0.
	\]
\end{proof}

	Under the assumptions of Theorem \ref{teorema}, for all $\gamma < \gamma_0$, Problem \eqref{discretesaddle} admits then a unique solution $\uh,\ph$, and, once again, C\'ea's lemma implies that there exists a constant $C_\gamma$, independent of $\U$ and $\Q$ and depending on $\comput$ only via the constants $c_\star$ and $C_\star$, such that
	\[
	\| \ucont - \uh \|\uno + \vvvert \pcont - \ph \vvvert\uno \leq C_\gamma \left ( \inf_{v \in \U} \| \ucont - v \|\uno + \inf_{q \in \Q} \vvvert \pcont - q \vvvert\uno \right).
	\]

\begin{remark}
	The coercivity assumption \eqref{coerca} allowed us to reduce the stabilization from the full form in \eqref{bbsaddle} to the minimal stabilization term in \eqref{discretesaddlecont}. Remark however that also when only \eqref{coercker} holds, it is still possible to use a reduced form of the stablization. Indeed, thanks to  \eqref{coercker}, the stabilization term does not need to control $A_a u$, but, to get control on $u$, it only needs to control $B^T u$ in the $\H'$ norm. We can then replace the stabilization bilinear form $\stab$ defined in \eqref{deffullstab} with the following one
	\[
	\stab(u,p;v,q) = -\dualscal{Au - B p}{Bq} + \lscalH B^T u , B^T v \rscalH\menouno,
	\]
	and modify the right hand side accordingly. Observing that, under our assumptions, we have that, for some positive constant $\alpha'$
	\[
\alpha'	\| v \|\uno^2 \leq a(v,v) + \vvvert B^T v \vvvert\menouno^2 \qquad \forall v \in \V,
	\]
	it is not difficult to prove that, for $\gamma$ sufficiently small, $\bil + \gamma\stab$ is, also for this definition of $\stab$, coercive on $\V \times \H$. For other ways of reducing or modifying the form of the stabilization, see \cite{BoffiBrezziFortin}.
\end{remark}

	\begin{remark}\label{rem:meshsize}
		In general, the two finite dimensional  approximation spaces $\U$, $\Q$ will depend on a discretization parameter $h$ (such as the meshsize or the granularity), going to zero as $\U$ and $\Q$ tend to dense subsets of $\V$ and $\H$. The dependence of all the inequalities on  such a parameter  (or their independence thereof) is here implicitly included in the constants appearing in the different assumptions, which, in general, might (or might not) depend on $h$ and possibly explode or go to zero as $h$ goes to zero. We refer, in particular, to the constants $\Cstar$ and $\cstar$ (remark that the constants $\alpha$, $\CA$, $\widehat c$ and $\widehat C$ depend on the continuous problem and are therefore always independent of $h$). Whenever the bilinear form $\mathbf{c}$ is constructed in such a way that these constants are independent of $h$, $\hbg$ is also independent of $h$. 
	\end{remark}

% ----------------------------------------------------------------
\section{Equivalent inner product for dual space}\label{dualscalprod}

We now consider the problem of defining bilinear forms that can serve as an equivalent scalar product for finite dimensional subspaces of the dual space $\V'$. 
 To this aim, we start by considering an auxiliary  finite dimensional subspace $W \subset \V$ of the primal space, endowed with a basis $\basis = \{ \e_n, \ n = 1,\cdots,N \}$.
 Let $P: \V \to W$ be a bounded linear projector, and let $P^*: \V' \to \tW$, with $\tW  =  \mathfrak{Im}(P^*) \subset \V'$, denote its adjoint, which is itself a projector ($ \mathfrak{Im}(P^*)$ stands here for the image of the operator $P^*$).
The following proposition, where $\delta$ is the Kronecker delta function, holds.

\begin{proposition}\label{prop:dualspace}
There exists a basis $\basis^*=\{ \te_n,\ n=1,\cdots,N \}$ for the space $\tW $ which verifies a biorthogonality property of the form
\begin{equation}\label{biorth}
\langle \te_n , \e_\ell \rangle = \delta_{n,\ell}, \qquad 1 \leq n,\ell \leq N.
\end{equation}
Moreover for all $u \in \V$, $\zeta \in \V'$ we have
\begin{equation}\label{projections}
P u = \sum_{n=1}^N \langle \te_n , u \rangle  \e_n,  \qquad P^* \zeta = \sum_{n=1}^N \langle  \zeta , e_n \rangle \te_n,
\end{equation}
where the first part of \eqref{projections} is an identity in $\V$ and the second an identity in $\V'$.
\end{proposition}

\begin{proof} As $W$ is finite dimensional, and as all norms on finite dimensional spaces are equivalent, we know that  there exist two positive constants $0< c \leq C$  such that  for all elements $v = \sum_{n=1}^N v_n \e_n$ we have
	\begin{equation}\label{Riesz}c \| v \|\uno \leq \sum_{n=1}^N | v_n | \leq C \| v \|\uno.\end{equation} 
	As $P$ is bounded, this implies that for each $n$ the functional  which maps a function $w \in \V$ to the coefficient $w_n$ in the expansion of $Pw = \sum_{n=1}^N w_n e_n$ with respect to the basis $\basis$ is linear and bounded, and then, by the Riesz's representation Theorem, there exists an element $\te_n$ such that  $w_n= \langle \te_n,w \rangle$ for all $w \in \V$. As $P$ is a projector, we easily verify that both \eqref{biorth} and the first equality in \eqref{projections} hold.
	Furthermore, we have that for $\zeta \in \V'$ and $w \in \V$, by the definition of the adjoint
	\begin{gather*}
	\langle P^* \zeta,w \rangle = \langle \zeta , P w \rangle = \langle \zeta , \sum_{n=1}^N \langle \te_n , w \rangle \e_n \rangle %\\[1mm] =	\sum_{n=1}^N  \langle \zeta, \e_n \rangle  \langle \te_n , w \rangle 
	= 
\langle 	\sum_{n=1}^N \langle \zeta, \e_n \rangle \te_n, w \rangle.
	\end{gather*}
As this identity holds for all $w \in \V$,  the second equality in \eqref{projections} is also proven.
	\end{proof}
	
 The space $W^*$ depends on the projector $P$. Proposition \ref{prop:dualspace} gives us then a class of bases of finite dimensional subspaces $W^* \subset \V'$, one for each bounded linear projector $P: \V \to W$.
		Remark that the constants $c$ and $C$ in the norm equivalence \eqref{Riesz} do generally depend on the projector $P$ and on the basis $\basis$. However the actual value of the two constants does not play a role in the identities \eqref{biorth} and \eqref{projections}.

		\newcommand{\Port}{P}
		\newcommand{\Riesz}{\Phi}
		
		\newcommand{\Gramian}{\matr{G}}
\
\begin{remark}
			Among the possible projectors $P$, we have the orthogonal projector, defined as
			\[
\Port u \in W, \qquad		( \Port u , w) = (u , w) \qquad \text{for all } w \in W.
			\]
			For such a choice we have $ W^* = \Riesz^{-1}(W)$, where $\Riesz:\V' \to \V$ is the Riesz operator, which, we recall, is defined by
			\[		
\Riesz \eta \in \V, \qquad			(\Riesz \eta,w) = \langle \eta , w \rangle \qquad \text{for all } w \in\V,
			\]
			(by the Riesz's representation theorem such an operator is an isomorphism between $\V'$ and $\V$).
			 Letting $\Gramian = (g_{nk})$ denote the Gramian matrix 
			\[
			\Gramian = (g_{nk}), \quad g_{nk} = (e_n , e_k),
			\]
			the basis $\basis^*$ is then, in this case, defined as
			\[
			\eta_n = \sum_{i=1}^N b_{ni} \Riesz^{-1}(e_i), \quad \text{with the matrix } \matr{B} = (b_{nk}) \text{ defined by } \matr{B} = \Gramian^{-1}.
			\]
		\end{remark}
		
		\
		
From now on we will make use of the following notational convention: we will use roman letters for the elements of $W$ and greek letters for the elements of $\tW $, and for $v = \sum_{n=1}^N v_n \e_n \in W$  and $\zeta = \sum_{n=1}^N \zeta_n \te_n \in \tW $ we will, respectively, denote the corresponding vector of coefficients by
$\vec v = (v_n) \in \mathbb{R}^N$  and $\vec \zeta = (\zeta_n) \in \mathbb{R}^N$. Thanks to the biorthogonality property \eqref{biorth}, the proof of the following proposition is straightforward.

\begin{proposition}\label{prop:2.4}
For $w \in W$, $\eta \in \tW $ we have $\langle \eta, w\rangle = \vec \eta^T \vec w$.
\end{proposition}

\

Let now $\matr{S}\in \mathbb{R}^{N\times N}$ be a symmetric  positive definite  matrix,   playing the role of a ``stiffness'' matrix, and such that  the bilinear form defined by
\[
s(v,w) = \vec w^T \matr{S} \vec v,
\]
for all $v, w \in W$, satisfies 
\begin{equation}\label{propertiess}
s(v,w) \leq \contS \| v \|\uno \| w \|\uno, \qquad s(v,v) \geq \coercS \| v \|\uno^2,
\end{equation}
with $\contS$, $\coercS$ positive constants.
We will indicate by   $S: W \to \tW $ the operator defined by
\[
\eta = S w \qquad \Leftrightarrow \qquad \vec \eta = \matr{S} \vec w.
\]

\begin{proposition} We have that, for all $w \in W$
\[
\| S w \|\menouno \leq \contS \normP \| w \|\uno.
 \]
 with $\normP = \sup_{v \in \V} \| P v \|\uno/\| v \|\uno.$
\end{proposition}

\begin{proof}
We have, with $u = \sum_{n=1}^N u_n \e_n =  Pv$,  
\begin{align*} 
    \| S w \|\menouno &= \sup_{v \in \V} \frac { \langle  Sw , v \rangle }{\| v \|\uno} = \sup_{v \in \V} \frac { \langle  S w , P v \rangle }{\| v \|\uno} \\[1mm] &= 
\sup_{v \in \V}
  \frac  { \vec u^T \matr{S} \vec w }{\| v \|\uno} \leq \frac{ \contS \| w \|\uno  \|  u \|\uno}{\| v \|\uno} \leq \contS \normP \| w \|\uno,
\end{align*}
where we used that, since $Sw \in \tW$ and since $P^*$ is a projector, we have $Sw = P^* Sw$.

\end{proof}

\

We now define a bilinear form $\widetilde s: \tW  \times \tW  \to \mathbb{R}$ as
\begin{equation}\label{defstilde}
\widetilde s(\eta,\zeta) = \vec \eta^T \matr{S}^{-1} \vec \zeta.
\end{equation}
As $\matr{S}^{-1}$ is symmetric positive definite, $\widetilde s$ is a scalar product on $\tW$, and it is possible to prove that it induces on $W^*$ a norm   equivalent to $\| \cdot \|\menouno$. More precisely we have the following proposition.
\begin{proposition}\label{prop:2.7}
For $\eta, \zeta \in \tW$ it holds that
\begin{equation}\label{eq2.6}
\widetilde s (\eta,\zeta) \leq \coercS^{-1} \| \eta \|\menouno \| \zeta \|\menouno, \qquad \widetilde s (\eta,\eta) \geq \normP^{-2} \contS^{-1} \| \eta \|^2\menouno.
\end{equation}
\end{proposition}

\begin{proof} As $\matr{S}^{-1}$ is positive definite, we have
\begin{gather*}
  \vec \eta^T\matr{S}^{-1} \vec \zeta \leq \sqrt{\vec \eta^T \matr{S}^{-1} \vec \eta}
   \sqrt{\vec \zeta^T \matr{S}^{-1} \vec \zeta}.
\end{gather*}
We now observe that, with $\vec \xi = \matr{S} \vec v$, we can write 
\begin{multline*}
\sqrt {\vec \eta^T \matr{S}^{-1} \vec \eta} = \frac{\vec \eta^T \matr{S}^{-1} \vec \eta}{\sqrt {\vec \eta^T \matr{S}^{-1} \vec \eta} }
 \leq \sup_{\vec \xi \in \mathbb{R}^N} \frac{\vec \eta^T \matr{S}^{-1} \vec \xi}{\sqrt {\vec\xi^T \matr{S}^{-1}\vec\xi}}= \sup_{\vec v \in \mathbb{R}^N} \frac{\vec \eta^T \matr{S}^{-1} \matr{S} \vec v}{\sqrt {\vec v^T \matr{S} \matr{S}^{-1} \matr{S}\vec v}} \\ =
\sup_{\vec v \in \mathbb{R}^N} \frac{\vec \eta^T \vec v}{\sqrt {\vec v^T \matr{S} \vec v}}
 \leq \coercS^{-1/2} \sup_{v \in W} \frac{\langle \eta,v \rangle }{\| v \|\uno} \leq \coercS^{-1/2} \| \eta \|\menouno,
\end{multline*}
where we used Proposition \ref{prop:2.4}. The same bound holds for $\vec \zeta$ so that we get the first bound in \eqref{eq2.6}.
On the other hand, for $\eta \in \tW $ we have, with $u = \sum_n u_n \e_n =  Pv$, and using Proposition \ref{prop:2.4},
\begin{multline*}
\| \eta \|\menouno 
= \sup_{v \in \V} \frac{\langle \eta, v\rangle}{\| v \|\uno} 
= \sup_{v \in \V} \frac{\langle P^*  \eta,  v\rangle}{\| v \|\uno} 
= \sup_{v \in \V} \frac{\langle \eta, P v\rangle}{\| v \|\uno} 
= \sup_{v \in \V} \frac{ \vec \eta^T \vec u} {\| v \|\uno} \\[2mm]
=  \sup_{v \in \V} \frac{ ( \matr{S}^{-1/2}\vec \eta)^T \matr{S}^{1/2}\vec u} {\| v \|\uno} 
 \leq \sup_{v \in \V} \frac{\sqrt{\vec \eta^T \matr{S}^{-1} \vec \eta}\, \sqrt{\vec u^T \matr{S}\vec u}}{\| v \|\uno} \\[1mm] \leq \sqrt{\contS} \sup_{v \in \V} \frac{\sqrt{\vec \eta^T \matr{S}^{-1} \vec \eta}\, \| Pv \|\uno}{\| v \|\uno} \leq \normP \sqrt{\contS} \sqrt{\vec \eta^T \matr{S}^{-1} \vec \eta}.
\end{multline*}
whence, by squaring, we get the desired result.
\end{proof}

Remark that if $P$ is chosen to be the orthogonal projection, we have the following corollary.

\begin{corollary}
For $\eta, \zeta \in \tW = \Phi^{-1}(W)$ it holds that
\[
\widetilde s (\eta,\zeta) \leq \coercS^{-1} \| \eta \|\menouno \| \zeta \|\menouno, \qquad \widetilde s (\eta,\eta) \geq \contS^{-1} \| \eta \|^2\menouno.
\]
\end{corollary}

%
%\newcommand{\hgstab}{\widehat G_\gamma}
%
%\begin{multline}\label{fulldiscretesaddle}
%\text{find $(\uh,\ph) \in U \times Q$ such that }\\
%\hag(\uh,\ph;v,q) = \langle F , v \rangle + \langle \hgstab, q \rangle \quad \forall v,q \in \U\times \Q
%\end{multline}

\section{Computable stabilization of the saddle point problem}\label{sec:5}

We focus now on the problem of defining a bilinear form  $\comput$ satisfying \eqref{contcomput} and \eqref{star}. We take advantage of the results presented in the previous section and start by selecting an auxiliary finite dimensional subspace $W \subset \V$ and a projector $P: \V \to W$. We then  
	%which, in view of Remark \ref{choiceprojector} we can assume to be the $\V$ orthogonal projection, 
	  let the bilinear form $\comput: \V' \times \V' \to \mathbb{R}$ be defined  as 
	\begin{equation}\label{defc} 
	\comput(\eta,\zeta) = \widetilde s(P^*(\eta),P^*(\zeta)),
	\end{equation}
	where $\widetilde s$ is defined by \eqref{defstilde}.
	It is not difficult to prove that the bilinear form $\comput$ is continuous, that is that  for all $\eta,\zeta \in \V'$,  it holds that
	\begin{equation}\label{contstar}
	\comput(\eta,\zeta) \leq \Cstar \| \eta \|\menouno  \| \zeta \|\menouno \qquad \text{with} \qquad \Cstar = \coercS^{-1} \| P \|^2\uno.
	\end{equation} 
Before tackling the issue of how the auxiliary space $W$ can be chosen so that \eqref{star} holds, let us face the problem of evaluating $\comput(Au - Bp,Bq)$, for $u$ in $\U$, and  $p, q$ in $\Q$. 
	We recall that, in the implementation of the discrete problem, the quantities that one usually has direct access to, are the  coefficient vectors $\vec x=(x_k)$, $\vec y = (y_\ell)$ and $\vec z = (z_\ell)$ of the expansion of, respectively, $u$ in a given basis
	$\BaseU = \{\basv_k,\ k = 1,\cdots, K  \}$ for $\U$,  and  $p$ and $q$ in a given basis $\BaseQ = \{\basf_\ell, \ \ell=1,\cdots,L \}$ for $\Q$:
	\begin{equation*}\label{development} u = \sum_{k=1}^K x_k \basv_k,\qquad 
	p = \sum_{\ell = 1}^L y_\ell \basf_\ell, \qquad  q = \sum_{\ell = 1}^L  z_\ell \basf_\ell.
	\end{equation*}
	A simple calculation yields \[
	P^*(Au) = \sum_{n=1}^N \langle Au , \e_n \rangle \eta_n = \sum_{n=1}^N \Big(\sum_{k=1}^K a(\basv_k,e_n)  x_k \Big)\eta_n,
	\]
	as well as
	\[
	P^*(Bp) = \sum_{n=1}^N\langle Bp , \e_n \rangle \eta_n = \sum_{n=1}^N\Big(\sum_{\ell = 1}^L b(\basf_\ell,e_n)  y_\ell \Big)\eta_n.
	\]
	Letting then $\widehat{\mathbf{A}}=(a_{n,k})$ and $\widehat{\mathbf{B}} = (b_{n,\ell})$ with
	\[
	a_{n,k} = a(\basv_k,e_n), \qquad b_{n,\ell} = b(\basf_\ell,e_n),
	\]
	it is not difficult to check that 
	\[
	\comput(Au-Bp,Bq) = (\widehat{\mathbf A} \vec x -  \widehat{\mathbf{B}} \vec y)^T \mathbf{S}^{-1} \widehat{\mathbf{B} }\vec z. \]

	It is important to observe that the basis $\basis^*$ of the auxiliary space $\tW $ is never used in the computation of $\comput(Bp,Bq)$. As a results, an explicit knowledge of such a basis is not necessary for implementing the method, which turns then out to be independent of the choice of the projector $P$, though the existence  of the latter and its properties are needed in order to perform the theoretical analysis of the method. We can then always assume that $P$ is the $\V$ orthogonal projection onto $W$, and that $\| P \|\uno = 1$.

\	
	
%	\subsection{The choice of the auxiliary space $W$}\label{sec:mangiacoda}
	We now aim at giving necessary and sufficient conditions on $W$ for \eqref{star} to hold. In view of the above observation, we can choose $P$ as the $\V$ orthogonal projection onto $W$. We have the following lemma.
	
	\newcommand{\hatalfa}{\widehat \alpha}
	
	\begin{lemma}
		Let $\Theta  \subset \V'$ be a finite dimensional subspace, and let $\comput$ be defined by \eqref{defc}, with $P$ chosen as the $\V$ orthogonal projection onto $W$. If	\begin{equation}\label{4}
		\inf_{\theta \in \Theta} \sup_{w \in W} \frac {\langle \theta,w \rangle}{\| \theta\|\menouno \| w \|\uno} \geq \hatalfa> 0,
		\end{equation}
		then, for $\cstar = \hatalfa^2 \contS^{-1}$, it holds that
		\begin{equation}\label{coerccTh}
	\cstar \| \theta \|^2 \leq	\comput(\theta,\theta), \qquad \forall \theta \in \Theta.
		\end{equation}
Conversely, if \eqref{coerccTh} holds, then \eqref{4} holds with $\hatalfa = (\cstar \coercS)^{1/2}$.
	\end{lemma}
	
	\begin{proof}
		Clearly, Proposition \ref{prop:2.7}  implies that  for all $p, q \in \H$ we have
		\begin{equation}\label{stellina}
		\comput(\theta,\theta) = \widetilde s(P^*(\theta),P^*(\theta)) \geq  \contS^{-1} \|  P^*(\theta) \|\menouno^2.\end{equation}
		
		Let us first prove that (\ref{4}) is a sufficient condition for \eqref{coerccTh} to hold. Let $\theta \in \Theta$. By (\ref{4}) we have
		\begin{gather*}
		\hatalfa \| \theta \|\menouno \leq \sup_{w \in W} \frac{\langle \theta , w \rangle}{\| w  \|\uno} 
		=  \sup_{w \in W} \frac{\langle \theta , P w \rangle}{\| w  \|\uno} 
		= \sup_{w \in W} \frac{\langle P^* (\theta ), w \rangle}{\| w  \|\uno}  \leq
		\| P^*( \theta) \|\menouno,
		\end{gather*}
		which, by \eqref{stellina}, yields
		\[
		\hatalfa^2 \| \theta \|^2\menouno \leq  \| P^*(\theta) \|\menouno^2 \leq  \contS \comput(\theta,\theta).
		\]

		\
		
		Let now assume that \eqref{coerccTh} holds, and let, once again, $\theta \in \Theta$. Using Proposition \ref{prop:2.7}, we can write
		\begin{equation}\label{defc}
		\cstar   \| \theta \|^2\menouno \leq \comput(\theta,\theta) =  \widetilde s(P^*(\theta),P^*(\theta)) \leq \coercS^{-1} \| P^*(\theta) \|\menouno^2.
		\end{equation}
		
		Now we can write
		\begin{gather*}
		\|  P^* (\theta) \|\menouno = \sup_{v \in \V} \frac {\langle  P^* (\theta) , v \rangle} {\| v \|\uno }=
		\sup_{v \in \V} \frac {\langle\theta ,  P  (v)\rangle} {\| v \|\uno}
		\leq \sup_{v \in \V} \frac {\langle\theta ,  P  (v) \rangle} {\| P   (v)\|\uno }= 
		\sup_{w \in W} \frac {\langle\theta ,  w\rangle} {\| w \|\uno}.
		\end{gather*}
		This yields, for all $\theta \in \Theta$,
		\[
		(c_\star \coercS)^{1/2} \| \theta \|\menouno \leq  \sup_{w \in W} \frac {\langle\theta ,  w\rangle} {\| w \|\uno}.
		\]
	The inf-sup bound \eqref{4} easily follows.
	\end{proof}
	
	\

	{
		We now observe that the inf-sup condition \eqref{infsupbb} implies  that $B: \H \to \mathfrak{Im}(B) \subset \V'$ is continuously invertible, and that for all $q \in \H$  we have
		\begin{equation}\label{boundB}
		\cinfsup \vvvert q \vvvert\uno \leq \| B q \|\menouno \leq \CB \vvvert q \vvvert\uno.
		\end{equation}
		In fact, the upper bound is simply the statement of the boundedness of $B$, while, by the definition of $ \mathfrak{Im}(B)$, for all $\eta \in  \mathfrak{Im}(B)$ it exists $q \in \H$ with $B q = \eta$. Moreover, given $q \in \H$, we can write
		\[
		\cinfsup \vvvert q \vvvert\uno \leq \sup_{v \in \V} \frac{b(q,v) }{\| v \|\uno} = \sup_{v \in \V}  \frac{\langle Bq , v \rangle}{\| v \|\uno} \leq \| Bq \|\menouno.
		\]
	}
	Then, it is not difficult to prove the following Proposition.
	
	\begin{proposition}
		If \eqref{infsupbb} holds, then we have
		\[
		\CB^{-1}\inf_{q\in \Q} \sup_{v \in W} \frac{b(q,v)}{\vvvert q \vvvert\uno\| v \|\uno } \leq \inf_{\theta\in B\Q} \sup_{w \in W} \frac{\langle \theta, v \rangle}{\| \theta \|\menouno \| v \|\uno }  \leq \cinfsup^{-1}\inf_{q\in \Q} \sup_{v \in W} \frac{b(q,v)}{\vvvert q \vvvert\uno \| v \|\uno }.
		\]
	\end{proposition}
	\begin{proof}
		Let  $\theta = B q \in \Theta = B\Q$. We have that
		\[
		\sup_{v \in W} \frac{\langle \theta, v \rangle}{\| \theta \|\menouno \| v \|\uno } = \sup_{v \in W} \frac{\langle B q,  v \rangle}{\| B q \|\menouno \| v \|\uno } =  \sup_{v \in W} \frac{ b(q,  v )}{\| B q \|\menouno \| v \|\uno }.\]
		Thanks to \eqref{boundB} we have that 
		\[
		\CB^{-1}\sup_{v \in W} \frac{ b(q,  v )}{\vvvert q \vvvert\uno \| v \|\uno } \leq  \sup_{v\in W} \frac{ b(q,  v )}{\| B q \|\menouno \| v \|\uno } \leq \beta^{-1} \sup_{v \in W} \frac{ b(q,  v )}{\vvvert q \vvvert\uno \| v \|\uno },
		\]
		which yields the desired result.
	\end{proof}
	
	\newcommand{\constinfsup}{\widehat \beta}
	
	\begin{corollary}\label{cor3.7} Let $\comput$ be defined by \eqref{defc}. Then,
		necessary condition for \eqref{star} to hold is that
		\begin{equation}\label{infsupbQW}
		\inf_{q\in \Q} \sup_{w\in W} \frac{b(q,w)}{\vvvert q \vvvert\uno \| w \|} \geq \constinfsup .
		\end{equation}
		with $ \constinfsup =  \beta (c_\star \coercS)^{1/2}$. Conversely, if \eqref{infsupbQW} holds, then \eqref{star} holds with $(c_\star \contS)^{1/2} = \CB^{-1} \constinfsup$.
	\end{corollary}
	
 The bad news stemming from Corollary \ref{cor3.7} is that, in order for assumption \eqref{star} to be satisfied uniformly in $h$, the space $W$ has to be chosen in such a way that a compatibility condition of the form \eqref{infsupbQW} holds. This is the kind of condition that, if satisfied by $\U$, would guarantee stability and optimality of the plain Galerkin discretization of Problem \eqref{contsaddle}. The good news is that the space $W$ is not required to satisfy any approximation property. In a way the present approach allows to uncouple the requirements relative to stability from those relative to approximation. We can then think of a class of schemes obtained by choosing three finite dimensional subspaces: $\U \subset \V$, $\Q \subset \H$ and $W \subset \V$. We ask of $\U$ and $\Q$  that they satisfy some approximation property. Usually this means that, if  $\ucont$ and $\pcont$ respectively belong to some subspace $\V^s \subset \V$ and $\H^s \subset \H$, endowed with stronger norms $\| \cdot \|_{s}$ and $\vvvert \cdot \vvvert_{s}$, we have the following approximation estimate, where $h$ is a mesh size type parameter in the spirit of Remark \ref{rem:meshsize},
	\[
	\inf_{v \in \U} \| \ucont - v \| \leq C h^s \| u \|_{s}, \qquad \inf_{q \in \Q} \vvvert \pcont - q \vvvert \leq C h^s \vvvert p \vvvert_{s},
	\]
At the same time, we only ask of $W$ that it satisfies, with $\Q$, an inf-sup condition of the form \eqref{infsupbQW}.

\

\newcommand{\gt}{\widetilde \gamma_0}
\renewcommand{\AA}{{\mathbf{a}}_\gamma}
\newcommand{\BB}{B}
\newcommand{\FF}{F}

\section{A different interpretation: decoupling approximation and compatibility }\label{sec:decoupling}
 It is interesting to look at the stabilized scheme as resulting, by static condensation, from the discretization of a problem where the auxiliary space $W$ is used to approximate an independent unknown.  
Let us, for simplicity, assume that the matrix $\matr{S}$ is obtained by  Galerkin projection from a  
bilinear form $s$, defined over the whole $\V \times  \V$ and verifying  \eqref{propertiess} for all $v,w \in \V$. We can consider the following continuous problem:
\begin{multline}\label{equivalent}
\text{find $(\ucont,\wcont) \in \V \times \V$,  $\pcont \in \H$ such that for all $(v, z) \in \V \times \V$,  $q \in \H$}\\[1mm]
\AA(\ucont,\wcont;v,z) - b (\pcont, v+z) + b(q,\ucont+\wcont) \\[1mm] = \langle F, v+z \rangle + \langle G, q \rangle
\end{multline}
with
\begin{align*}
\AA(u,w;v,z) &=	a(\ucont,v + z)  + \frac 1 \gamma s(\wcont,z).
\end{align*}

It is not difficult to realize that such a problem is well posed, and that it is equivalent to our original continuous problem. More precisely, the following proposition holds.
\begin{proposition}
	There exists $\gt$ such that for all $\gamma < \gt$, Problem \eqref{equivalent} admits the unique solution $(\ucont,\wcont,\pcont)$, with $\wcont=0$ and $\ucont,\pcont$ satisfying \eqref{contsaddle}.
\end{proposition}

\begin{proof}
	We have, for $\varepsilon > 0$ arbitrary
	\[
	\AA(u,w;u,w) \geq \left(\alpha - \frac{\varepsilon \CA} 2\right) \| u \|^2 + \left(\frac \coercS \gamma - \frac {\CA}{2 \varepsilon}\right) \| w \|^2.
	\]
	Choosing $\varepsilon = \alpha / \CA$, for $\gamma < \gt = 2 \coercS \alpha / \CA^2$ we have that 
\[
\AA(u,w;u,w) \geq \widecheck \alpha_\gamma(\| u \|^2 + \| w \|^2), \quad \text{with } \widecheck \alpha_\gamma = \min\left\{ \frac \alpha 2,    \frac \coercS \gamma - \frac{\CA^2}{2\alpha} \right \} > 0.
\]
Since $b$ satisfies \eqref{infsupbb}, existence and uniqueness of the solution immediately follow. As it is immediate to check that $(\ucont,\wcont,\pcont)$, with $\wcont=0$ and $\ucont,\pcont$ satisfying \eqref{contsaddle}, is a solution, the Proposition is proven.
	\end{proof}
	
	\newcommand{\wh}{w^\flat_h}

Let now $\U \subset \V$, $W \subset \V$ and $\Q \subset \H$ be finite dimensional subspaces. We can consider the discrete problem 
\begin{multline}\label{equivalent_discrete}
\text{find $(\uh,\wh) \in \U\times W$,  $p_h^\flat  \in \Q$ such that for all $(v,z) \in \U\times W$,  $q \in \Q$}\\[1mm]
\AA(\uh,\wh;v,z)- b(p_h^\flat ,v + z) + b(q,u_h^\flat +w_h^\flat ) \\[1mm] = \langle F, v + z \rangle + \langle G , v \rangle 
\end{multline}
Assume that $W$ and $\Q$ are compatible, in the sense that they satisfy an inf-sup condition of the form \eqref{infsupbQW}. Then, since 
\[
\inf_{q\in \Q} \sup_{(v,z) \in \U \times W} \frac{ b(q,v+z) }{\vvvert q \vvvert\uno(\| v \|\uno + \| z \|\uno)} \geq \inf_{q\in \Q} \sup_{z \in W} \frac{  b(q,z) }{\vvvert q \vvvert\uno \| z \|\uno} \geq \widehat{\beta},
\]
problem \eqref{equivalent_discrete} is well posed. Its algebraic form has the following structure
\[
\left(
\begin{array}{ccc}
\mathbf{A} & 0 &-\mathbf{B}\\
\widehat{\mathbf{A}}^T  &\frac 1 \gamma \mathbf{S} & -\widehat{\mathbf {B}} \\
\mathbf{B}^T & \widehat{\mathbf{B}}^T & 0
\end{array}
\right)\left(
\begin{array}{c}
\vec x \\
\vec z \\
\vec y
\end{array}\right) = \left(
\begin{array}{c}
\vec f \\
\vec {\widehat f} \\
\vec y
\end{array} 
\right)
\]
By solving the second equation with respect to $\vec z$ and plugging the result in the first and last equation, we obtain the stabilized formulation presented in the previous section. Remark that, as the $W$ component of the continuous solution is $0$, which belongs to $W$ for any choice of $W$, we confirm that the approximation properties of such an auxiliary space are immaterial for the accuracy of the method.

\

This interpretation of the stabilized method, suggests that the inf-sup condition \eqref{infsupbQW} can be relaxed. Indeed for the discrete problem \eqref{equivalent_discrete} to be well posed, it is sufficient that 
\[
\inf_{q \in \Q} \sup_{w \in \U + W} \frac {b (q,w)}{\vvvert q \vvvert\uno\| w \|\uno} \geq \constinfsup.
\]

\section{Conclusions}\label{sec:conclusions}
In an abstract framework, we presented and analysed a systematic approach to realize computable scalar products for dual spaces. Based on this approach we designed a residual based stabilization technique for saddle point problems which, in principle, allows to circumvent the discrete inf-sup condition required for stability and optimal convergence of the discretization of such a class of problem. The method relies on the introduction of a third, auxiliary, subspace. The theoretical analysis of the resulting method shows that, while the approximation spaces for primal variable and Lagrange multiplier can indeed be chosen independently of each other,  the latter and the auxiliary space must satisfy the same kind of inf-sup condition that one is trying to circumvent in the first place. However, the requirement on the auxiliary space are weaker than the ones on the approximation space for the primal variable, as the former does not need to satisfy any approximation assumption. In a way, the present approach allows to decouple the requirements for stability, from the ones for approximation. We believe that this approach, which we already tested in the framework of Discontinuous Galerkin methods \cite{DGPoly1, Bertoluzza2020probust}, and of the non conforming Virtual Element Method \cite{bertoluzza2021stabilization}, might give a novel insight on the design of stabilization terms in those cases where the usual  SUPG--like stabilization terms lead to suboptimal results due, for instance, to the lack of suitable inverse inequalities.

% ----------------------------------------------------------------
\bibliographystyle{amsplain}
\bibliography{Bertoluzza}

\providecommand{\bysame}{\leavevmode\hbox to3em{\hrulefill}\thinspace}
\providecommand{\MR}{\relax\ifhmode\unskip\space\fi MR }
\providecommand{\MRhref}[2]{%
  \href{http://www.ams.org/mathscinet-getitem?mr=#1}{#2}
}
\providecommand{\href}[2]{#2}
\begin{thebibliography}{10}

\bibitem{Arioli}
M.~Arioli and D.~Loghin, \emph{Discrete interpolation norms with applications},
  SIAM J. Numer. Anal. \textbf{47} (2009), 2924–--2951.

\bibitem{MINI}
D.N. Arnold, F.~Brezzi, and M.~Fortin, \emph{A stable finite element for the
  {S}tokes equations}, Calcolo \textbf{21} (1984), 337–--344.

\bibitem{Babuska}
I.~Babuska, \emph{The finite element method with {L}agrange multipliers},
  Numer. Math. (1973).

\bibitem{BabuskaAziz}
I.~Babuska and A.K. Aziz, \emph{Survey lectures on the mathematical foundations
  of the finite element method}, Academic Press, 1972.

\bibitem{BBstab}
C.~Baiocchi and F.~Brezzi, \emph{Stabilization of unstable methods}, Problemi
  attuali dell'Analisi e della Fisica Matematica (P.E. Ricci, ed.),
  Universit{\`a} ``La Sapienza'', Roma, 1993.

\bibitem{VBG}
C.~Baiocchi, F.~Brezzi, and L.P. Franca, \emph{Virtual bubbles and
  galerkin-least-squares type methods (ga.l.s.)}, Computer Methods in Applied
  Mechanics and Engineering \textbf{105} (1993), no.~1, 125--141.

\bibitem{Babs2}
S.~Bertoluzza, \emph{Stabilization by multiscale decomposition}, Appl. Math.
  Lett. \textbf{11} (1998), no.~6, 129--134.

\bibitem{B:Lag}
\bysame, \emph{Wavelet stabilization of the {L}agrange multiplier method},
  Numer. Math. \textbf{86} (2000), 1--28.

\bibitem{BCT}
S.~Bertoluzza, C.~Canuto, and A.~Tabacco, \emph{Stable discretization of
  convection-diffusion problems via computable negative order inner products},
  SINUM \textbf{38} (2000), 1034--1055.

\bibitem{BK}
S.~Bertoluzza and A.~Kunoth, \emph{Wavelet stabilization and preconditioning
  for domain decomposition}, I.M.A. Jour. Numer. Anal. \textbf{20} (2000),
  533--559.

\bibitem{bertoluzza2021stabilization}
S.~Bertoluzza, G.~Manzini, M.~Pennacchio, and D.~Prada, \emph{Stabilization of
  the nonconforming {V}irtual {E}lement {M}ethod}, 2021.

\bibitem{Bertoluzza2020probust}
S.~Bertoluzza, I.~Perugia, and D.~Prada, \emph{A $p$-robust polygonal
  discontinuous {G}alerkin method with minus one stabilization}, Math. Mod.
  Meth. Appl. Sc \textbf{31}, no.~13, 2695--2731.

\bibitem{DGPoly1}
S.~Bertoluzza and D.~Prada, \emph{A polygonal discontinuous {G}alerkin method
  with minus one stabilization}, ESAIM: M2AN \textbf{55} (2021), S785--S810.

\bibitem{Bochevetal2006}
P.~B. Bochev and R.~B. Lehoucq, \emph{Regularization and stabilization of
  discrete saddle-point variational problems}, Electron. Trans. Numer. Anal.
  \textbf{22} (2006), 97--113.

\bibitem{BoffiBrezziFortin}
D.~Boffi, F.~Brezzi, and M.~Fortin, \emph{Mixed finite element methods and
  applications}, Springer-Verlag Berlin Heidelberg.

\bibitem{BolandNicolaides1985}
J.~M. Boland and R.~A. Nicolaides, \emph{Stable and semistable low order finite
  elements for viscous flows}, SIAM J. Numer. Anal. \textbf{22} (1985), no.~3,
  474--492.

\bibitem{Bramble98}
J.~H. Bramble, R.~D. Lazarov, and J.~E. Pasciak, \emph{Least-squares for
  second-order elliptic problems}, Computer Methods in Applied Mechanics and
  Engineering \textbf{152} (1998), no.~1, 195--210.

\bibitem{Bramble01}
\bysame, \emph{Least-squares methods for linear elasticity based on a discrete
  minus one inner product}, Computer Methods in Applied Mechanics and
  Engineering \textbf{191} (2001), no.~8, 727--744.

\bibitem{BFMR}
F.~Brezzi, L.~Franca, D.~Marini, and A.~Russo, \emph{Stabilization techniques
  for domain decomposition methods with non-matching grids}, Proc. IX Domain
  Decomposition Methods Conference.

\bibitem{SUPG}
A.~N. Brooks and T.~J.~R. Hughes, \emph{Streamline upwind/{P}etrov-{G}alerkin
  formulations for convection dominated flows with particular emphasis on the
  incompressible {N}avier-{S}tokes equations}, Comput. Methods Appl. Mech.
  Engrg. \textbf{32} (1982), 199--259.

\bibitem{Cea}
J.~Cea, \emph{Approximation variationelle des probl{\`e}mes aux limites},
  \textbf{14} (1964), 345--444.

\bibitem{DPG2}
L.~Demkowicz and J.~Gopalakrishnan, \emph{A class of discontinuous
  {P}etrov-{G}alerkin methods. part {II}: Optimal test functions},
  \textbf{27}, no.~1, 70--105.

\bibitem{DPG}
\bysame, \emph{Discontinuous {P}etrov–{G}alerkin ({DPG}) method}, pp.~1--15,
  John Wiley \& Sons, 2017.

\bibitem{GLS}
T.~J.~R. Hughes, L.~P. Franca, and G.~M. Hulbert, \emph{A new finite element
  formulation for computational fluid dynamics: Viii. the
  galerkin/least-squares method for advective-diffusive equations}, Comput.
  Methods Appl. Mech. Engrg. \textbf{73} (1989).

\bibitem{RozzaSupremizer}
A.~Quarteroni and G.~Rozza, \emph{Numerical solution of parametrized
  {N}avier–{S}tokes equations by reduced basis methods}, Numerical Methods
  for Partial Differential Equations \textbf{23} (2007), no.~4, 923--948.

\end{thebibliography}

%\section{Annexe}
%
%Let $\pi_N: L^2(]-1,1[) \to \mathbb{P}_{N+2} \cap H^1_0(]-1,1[)$ be defined by
%\[
%\int\menouno^1 \pi_N \phi(x) p(x) \,dx = \int\menouno^1 \phi(x) p(x) \,dx, \ \forall p \in \mathbb{P}_N
%\]
%Then for all $r,s$ with $0 \leq r \leq s$, $1/2 < s \leq 1$ 
%\[
%\| \phi - \pi_N \phi \|_{0,}
%\]

\end{document}